\documentclass[a4paper,draft]{amsart}
\usepackage{amsmath}
\usepackage{amssymb}
\usepackage{bbm}
\usepackage{a4wide}

\parskip\smallskipamount

\let\set\mathbbm
\def\Ann{\operatorname{Ann}}

\newtheorem{thm}{Theorem}
\newtheorem{lemma}{Lemma}

\begin{document}

 \author[Manuel Kauers and Christoph Koutschan and Doron Zeilberger]
    {Manuel Kauers\,$^\ast$ and Christoph Koutschan\,$^\ast$ and Doron Zeilberger\,$^\dagger$}

 \address{Manuel Kauers, Research Institute for Symbolic Computation, J. Kepler University Linz, Austria}
 \email{mkauers@risc.uni-linz.ac.at}

 \address{Christoph Koutschan, Research Institute for Symbolic Computation, J. Kepler University Linz, Austria.}
 \email{ckoutsch@risc.uni-linz.ac.at}

 \address{Doron Zeilberger, Mathematics Department, Rutgers University (New Brunswick), Piscataway, NJ, USA.}
 \email{zeilberg@math.rutgers.edu}

 \thanks{$^\ast$ Supported in part by the Austrian FWF
  grants SFB F1305, P19462-N18, and P20162-N18.}
 \thanks{$^\dagger$ Supported in part by the USA National Science Foundation.}

 \title{Proof of Ira Gessel's Lattice Path Conjecture}

 \begin{abstract}
   We present a computer-aided, yet fully rigorous,  proof of Ira Gessel's tantalizingly simply-stated
   conjecture that the number of ways of walking $2n$ steps in the 
   region $x+y \geq 0, y \geq 0$ of the square-lattice
   with unit steps in the east, west, north, and south directions, that start and end at the origin, equals 
   $16^n\frac{(5/6)_n(1/2)_n}{(5/3)_n(2)_n}$ .
 \end{abstract}


 \maketitle

 \def\gesselsteps{\{{\leftarrow},\penalty100{\rightarrow},\penalty100{\swarrow},\penalty100{\nearrow}\}}
 \def\krewerassteps{\{{\leftarrow},\penalty100{\downarrow},\penalty100{\nearrow}\}}

 \section{Introduction}

 There is a certain family of lattice walks, let's call them the Gessel walks, whose
 counting function is puzzling the combinatorialists already for several years.
 Gessel walks (that are trivially equivalent to the walks described in the abstract)
 are walks in the lattice $\set Z^2$ that stay entirely in the first
 quadrant (viz.\ they are walks in $\set N^2$) and that only consist of unit steps 
 chosen from $G:=\gesselsteps$.
 If $f(n;i,j)$ denotes the number of Gessel walks with exactly $n$~steps starting at the 
 origin $(0,0)$ and ending at the point $(i,j)$, then the counting function is 
 the multivariate power series
 \[
   F(t;x,y):=\sum_{n=0}^\infty\sum_{i=0}^\infty\sum_{j=0}^\infty f(n;i,j)x^i y^j t^n.
 \]
 We would call this power series holonomic (with respect to~$t$) if it satisfies an ordinary
 linear differential equation (with respect to~$t$) with polynomial coefficients in~$t,x,y$.
 This may or may not be the case.

 For example, the Kreweras walks are defined just like the Gessel walks, but with the unit steps
 chosen from $\krewerassteps$ instead of from~$G$. It is a classical result~\cite{kreweras65}
 that their counting function is holonomic.
 In constrast, Bousquet-M\'elou and Petkov\v sek showed that the counting function for certain 
 Knight walks is not holonomic~\cite{melou00}.
 Mishna~\cite{mishra06} provides a systematic study of all the possible walks in the quarter
 plane with steps chosen from any step set 
 $S\subseteq\{{\leftarrow},\penalty100{\nwarrow},\penalty100{\uparrow},\penalty100{\nearrow},\penalty100
 {\rightarrow},\penalty100{\searrow},\penalty100{\downarrow},\penalty100{\swarrow}\}$ 
 with $|S|=3$.
 She shows that the counting functions for the step sets 
 $\{{\nearrow},\penalty100{\searrow},\penalty100{\nwarrow}\}$ and 
 $\{{\nearrow},\penalty100{\searrow},\penalty100{\uparrow}\}$ (and some others that
 are equivalent to those by symmetry) are not holonomic while all others are holonomic.

 For the number of walks returning to the origin, there is sometimes a nice closed form representation,
 even if there is no such representation for the number of walks to an arbitrary point
 $(i,j)$. For instance, if $k(n;i,j)$ denotes the number of Kreweras walks of length~$n$
 from $(0,0)$ to $(i,j)$~\cite[A006335]{sloane95}, then~\cite{kreweras65}
 \[
   k(3n;0,0)=\frac{4^n}{(n+1)(2n+1)}\binom{3n}n\qquad(n\geq0)
 \]
 and $k(n;0,0)=0$ if $n$ is not a multiple of~$3$.
 
 Gessel~\cite{gessel} observed that a similar representation seems to exist for the number 
 $f(n;0,0)$ of Gessel walks returning to the origin~\cite[A135404]{sloane95}.
 He conjectured the following closed form representation.

 \begin{thm} Let $f(n;i,j)$ denote the number of Gessel walks going in $n$ steps from 
   $(0,0)$ to $(i,j)$. Then $f(n;0,0)=0$ if $n$ is odd and
   \[
    f(2n;0,0)=16^n\frac{(5/6)_n(1/2)_n}{(5/3)_n(2)_n}\qquad(n\geq0),
   \]
   where $(a)_n:=a(a+1)\cdots(a+n-1)$ denotes the Pochhammer symbol.
 \end{thm}

 The purpose of the present article is to describe, to human beings, the proof of this theorem.
 The proof is accomplished by computing a homogeneous linear recurrence in $n$ for $f(n;0,0)$. 
 Then the statement follows directly by verifying that the right hand side satisfies
 the same recurrence and that the initial values match.
 Our recurrence has order~32, polynomial coefficients of degree~172, and involves
 integers with up to 385~decimal digits.
 As this is somewhat too much to be printed here (it would cover about 250~pages),
 we provide it electronically at
 \begin{center}
   http://www.math.rutgers.edu/\~{}zeilberg/tokhniot/Guessel2
 \end{center}
 which has the recurrence as a Maple procedure whose output being $0$ proves Gessel's conjecture
 (once the first 32 initial values are checked, but this has already been done by Gessel
 himself, when he formuated his conjecture) .

 Our result implies that $F(t;0,0)$ is holonomic with respect to~$t$, but it has no
 direct implications concerning the holonomy of $F(t;x,y)$ for other $x,y$ of interest,
 e.g., $x=y=1$, and all the more so for general $x$,~$y$.
 There is, however, strong evidence that even the general counting function $F(t;x,y)$
 with ``symbolic'' $x$ and $y$ is holonomic in~$t$, see~\cite{bostan08}.


 \section{The Quasi-Holonomic Ansatz}

 \subsection{Annihilating Operators}

 Let $S_n,S_i,S_j$ be the shift operators, acting on $f(n;i,j)$ in the natural way,
 e.g., $S_nf(n;i,j)=f(n+1;i,j)$.
 An annihilating operator of~$f(n;i,j)$ is an operator~$R$ with
 \[
   R(n,i,j,S_n,S_i,S_j)f(n;i,j)=0.
 \]
 Those operators belong to a noncommutative polynomial algebra $\set Q(n,i,j)[S_n,S_i,S_j]$,
 and together they form a left ideal in that algebra, called the \emph{annihilator} of $f(n;i,j)$.
 Note that for an annihilating operator we can always assume polynomial coefficients 
 instead of rational coefficients, i.e., that the operator belongs to $\set Q[n,i,j][S_n,S_i,S_j]$,
 by clearing denominators.  

 Our goal is to find an annihilating operator for Gessel's $f(n;i,j)$ that implies 
 the conjecture. 
 For example, it would be sufficient to know an annihilating operator 
 $R(n,i,j,S_n)$ free of the shifts $S_i$ and~$S_j$, because then
 $R(n,0,0,S_n)$ would be an annihilating operator for $f(n;0,0)$.

 For this reasoning to apply, we could actually be less restrictive and allow
 also shifts in $i$ and $j$ to occur in~$R$, as long as they disappear when $i$ and $j$
 are set to zero. Our goal, therefore, is to find operators $P,Q_1,Q_2$ such
 that
 \[
   R(n,i,j,S_n,S_i,S_j)=P(n,S_n)+i Q_1(n,i,j,S_n,S_i,S_j)
                                +j Q_2(n,i,j,S_n,S_i,S_j)
 \]
 annihilates~$f(n;i,j)$. This is the \emph{quasi-holonomic ansatz}~\cite{kauers07v}.

 \subsection{Discovering annihilating operators}\label{Guess}

 We search for annihilating operators by making, for some fixed~$d$, an ansatz 
 \[
  R = 
  \sum_{0\leq e_1,\dots,e_6\leq d} c_{e_1,\dots,e_6} n^{e_1} i^{e_2} j^{e_3} S_n^{e_4} S_i^{e_5} S_j^{e_6}
 \]
 with undetermined coefficients $c_{e_1,\dots,e_6}$. Applying this ``operator template'' to 
 $f(n;i,j)$ gives
 \[
   \sum_{0\leq e_1,\dots,e_6\leq d} c_{e_1,\dots,e_6} n^{e_1} i^{e_2} j^{e_3} f(n+e_4;i+e_5,j+e_6),
 \]
 which, when equated to zero for any specific choice of $n,i,j$ yields a linear constraint for the
 undetermined coefficients. (Note that $f(n;i,j)$ can be computed efficiently for
 any given $n,i,j\in\set Z$.)

 By taking several different $n,i,j$ we obtain a linear system of equations.
 If that system has no solution, then there is definitely no annihilating operator matching
 the template. If there are solutions, then these are candidates for annihilating
 operators. 

 We can clearly restrict the search to quasi-holonomic operators by leaving out unwanted
 terms in the ansatz for~$R$.

 \subsection{Verifying conjectured annihilating operators}\label{VerifyGuess}

 An algorithm was given in~\cite{kauers07v} for deciding whether some given operator 
 $R\in\set Q(n,i,j)[S_n,S_i,S_j]$ annihilates $f(n;i,j)$ or not. We repeat this algorithm
 for the sake of self-containedness.

 First note that the step set $\{\leftarrow,\rightarrow,\nearrow,\swarrow\}$ gives readily
 rise to the recurrence
 \[
   f(n+1;i,j) = f(n;i+1,j) + f(n;i-1,j) + f(n;i+1,j+1) + f(n;i-1,j-1),
 \] 
 and therefore the ``trivial operator''
 \[
   T:= S_nS_iS_j - S_i^2S_j - S_j - S_i^2S_j^2 - 1
 \]
 certainly annihilates~$f(n;i,j)$.
 Instead of checking that $R$ annihilates $f(n;i,j)$, we will check that 
 $TR$ annihilates~$f(n;i,j)$. By the following lemma, this is sufficient.

 \begin{lemma}
   Suppose that an operator $R$ is such that $(TR)f(n;i,j)=0$.
   Then it can be checked whether $Rf(n;i,j)=0.$   
 \end{lemma}
 \begin{proof}
   $(TR)f(n;i,j)=0$ implies that $T$ annihilates $Rf(n;i,j)$, i.e., $Rf(n;i,j)$ also
   satisfies the above recurrence. Therefore, in order to show that $Rf(n;i,j)=0$
   entirely, it suffices to show that $Rf(n;i,j)=0$ for $n=0$ and all $i$ and~$j$.
   If $r_n$ bounds the degree of $S_n$ in $R$, then it suffices to verify
   $Rf(n;i,j)=0$ for $n=0$ and $0\leq i,j\leq r_n$, because we clearly have $f(n;i,j)=0$ 
   for $i>n$ or $j>n$. This leaves us with checking finitely many values,
   which can be done.
 \end{proof}

 By the lemma, in order to check $Rf(n;i,j)=0$, it suffices to be able
 to check $(TR)f(n;i,j)=0$. 
 For checking the latter, compute operators $U,V$ with $TR=UT+V$ by division with remainder. Then 
 \[
   (TR)f(n;i,j)=0 \iff Vf(n;i,j)=0.
 \]
 If $V$ is the zero operator, then
 we are done, otherwise we proceed recursively to show that $Vf(n;i,j)=0$
 (compute $U',V'$ with $TV=U'T+V'$, observe that $(TV)f(n;i,j)=0$ iff $V'f(n;i,j)=0$, and so on.)
 As $T$ has constant coefficients and, for any~$d>0$, the commutation rules 
 in $\set Q(n,i,j)[S_n,S_i,S_j]$ are such that
 $S_in^d=n^dS_i$, $S_jn^d=n^dS_j$ and $S_nn^d=n^dS_n+\mathrm{O}(n^{d-1})$
 (and similarly for $i$ and $j$ in place of~$n$),
 it follows that the degree of $V$ with respect to $n,i,j$ will be strictly smaller than
 the degree of $R$ with respect to these variables.
 Therefore, the recursion must eventually come to an end.

 \subsection{Nice idea, but\dots}

 At this point, we know that all we need for proving the conjecture is a quasi-holonomic
 annihilating operator for~$f(n;i,j)$. We know how to search for such operators, 
 and once empirically discovered, we know how to verify them. 

 Unfortunately, it turned out that if a quasi-holonomic annihilating operator for $f(n;i,j)$
 exists at all, then it must be quite large. It was shown~\cite{kauers07v} that there is no 
 such operator of order up 
 to~$8$ in either direction with polynomial coefficients of total degree at most~$6$. 
 Increasing the bounds on order and degree further might, of course, help, but this 
 is beyond our current computing capabilities. 
 (For the above assertion, a dense linear system with several thousand 
 variables and equations had to be solved exactly.)

 \section{A Takayama-Style Approach}

 By making an ansatz, we could not find a quasi-holonomic annihilating operator, but we could find
 (and verify) plenty of other operators, $R_1,R_2,R_3,\dots$ that were not of the quasi-holonomic type. 
 Once it has been verified that these $R_i$ are indeed annihilating operators, we may of course
 freely choose any operators $P_1,P_2,P_3,\dots$, and the combined operator
 \[
   P_1R_1+P_2R_2+P_3R_3+\cdots
 \]
 will again be an annihilating operator. In other words, all
 annihilating operators form a left ideal in the corresponding algebra. Our
 next step is to find a quasiholonomic combination of the operators
 $R_1,R_2,R_3,\dots$ that were found (and verified) by the method of
 the previous section.

 \subsection{Takayama's Algorithm}

 Assume we want to find a recurrence for the sum
 \[
  \sum_k f(k,n).
 \]
 In his ``holonomic systems approach'' \cite{zeilberger90} the third-named
 author proposes to search for an annihilating operator $R$ of
 $f(n,k)$ of the form
 \[
  R(n,S_k,S_n)=P(n,S_n)+(S_k-1)Q(n,S_k,S_n).
 \]
 Summing over $k$ shows (in case of natural boundaries which we will
 assume in the following) that $P(n,S_n)$ annihilates the sum.
 Starting with the annihilator of the summand~$f(n,k)$ in $A=\set
 Q(k,n)[S_k,S_n]$, i.e.  $\Ann_A f\subseteq A$, one computes an
 $R(n,S_k,S_n)\in\Ann_A f$ free of $k$ (e.g., by elimination via
 Gr\"obner bases). Any such $R$ can be brought to the desired form as
 above.

 Almkvist \cite{almkvist90} observed that in the above setting the
 constraint for $Q$ can be released: The whole proof would go through
 in the same way if additionally $Q$ depends on~$k$. This fact is exploited 
 in Takayama's algorithm \cite{takayama90,takayama90a} which
 originally was formulated only in the context of the Weyl algebra.
 Chyzak and Salvy \cite{chyzak98} extended the algorithm to more
 general Ore algebras (which include also the shift case that we are
 dealing with) and proposed some optimizations. The idea in short is
 the following: While in the algorithm above, first $k$ was eliminated
 and then the part $(S_k-1)Q$ was removed (which corresponds to divide
 out the right ideal $(S_k-1)A$), the order is now reversed.
 In Takayama's algorithm we first reduce modulo $(S_k-1)A$ and then
 perform the elimination of~$k$. The algorithm usually leads to
 shorter recurrences since we allow more freedom for~$Q$.  Second, the
 elimination is in general much faster since we got rid of~$Q$ from the
 very beginning. Note that $Q$ is not computed at all so we have to
 assure natural boundaries a priori.

 There is one technical complication in this approach. The fact that we
 are computing in a noncommutative algebra restricts us in the
 computations after having divided out the right ideal $(S_k-1)A$.
 In particular, we are no longer allowed to multiply by $k$
 from the left.  We can easily convince ourselves that otherwise we
 would get wrong results: Assume we have written an operator already
 in the form $P+(S_k-1)Q$. Multiplying it by $k$ and then reducing it
 by $(S_k-1)A$ leads to $kP-Q$ since we have to rewrite
 $k(S_k-1)$ as $(S_k-1)(k-1)-1$. Because $k$ does not commute with
 $S_k-1$ we get the additional term $-Q$ in the result which we lose
 if we first remove $(S_k-1)Q$ and then multiply by $k$.

 In order to find a $k$-free operator one needs an elimination
 procedure that avoids multiplying by~$k$. Let now
 $R_1,\dots,R_m\in A$ be the operators which generate~$\Ann_A f$,
 and let $R_1',\dots,R_m'\in\set Q(k,n)[S_n]$ be the corresponding
 reductions modulo $(S_k-1)A$.  For $1\leq i\leq m$ we can write
 $$R_i'(k,n,S_n)=R_{i,0}(n,S_n)+R_{i,1}(n,S_n)k+R_{i,2}(n,S_n)k^2+\dots$$
 where $R_{i,j}\in A':=\set Q(n)[S_n]$. Elimination of~$k$ now amounts to 
 finding a linear combination
 $$P_1(n,S_n)\left(\begin{array}{c}R_{1,0}\\ R_{1,1}\\ R_{1,2}\\ \vdots\end{array}\right)+\dots
  +P_m(n,S_n)\left(\begin{array}{c}R_{m,0}\\ R_{m,1}\\ R_{m,2}\\ \vdots\end{array}\right)
  =\left(\begin{array}{c}P(n,S_n)\\ 0\\ 0\\ \vdots\end{array}\right)$$
 for some $P_1,\dots,P_m\in A'$. The vector on the right hand side
 corresponds to the desired $k$-free operator. In general,
 this will not work yet since we can not expect to succeed in the
 elimination without multiplying by $k$ at all. Hence we also have 
 to include multiples of the $R_i$ by powers of $k$. More algebraically
 speaking, the computations take place in an $A'$-module $M$ that is
 generated by the above vectors plus all elements 
 $k^jR_i\!\mod (S_k-1)A,1\leq i\leq m,j=1,2,\dots$.
 The elimination is achieved by
 computing a Gr\"obner basis of this module. Note that $P\in M$ if and
 only if there exists a $Q\in A$ such that $P+(S_k-1)Q\in\Ann_A f$. For
 practical purposes we have to truncate the module $M$ by considering
 only elements up to a certain length $d$, i.e., which have zeros
 in all positions greater than $d$.
 The most natural choice for $d$ is the highest power $k^d$
 that appears in $R_1,\dots,R_m$. But we are not
 guaranteed that for any $P,Q\in A$ with $P+(S_k-1)Q\in\Ann_A f$ the
 operator~$P$ is an element of the truncated module.  In the unlucky
 case that no $k$-free operator was found, the bound $d$ has to be
 increased. The algorithm works similar in the case of multiple sums
 where we want to eliminate several variables $k_1,\dots,k_r$.

 \subsection{Proof of Gessel's conjecture}

 Now back to Gessel's conjecture: Recall that we were looking for a
 quasi-holonomic operator
 \[
   R(n,i,j,S_n,S_i,S_j)=P(n,S_n)+i Q_1(n,i,j,S_n,S_i,S_j)
                                +j Q_2(n,i,j,S_n,S_i,S_j)
 \]
 where we are mainly interested in $P$, because $Q_1$ and $Q_2$ anyway
 vanish when we set $i$ and $j$ to 0.  This task is very similar to
 the setting in the previous section and with slight modifications we
 can apply Takayama's algorithm to solve it.  The only difference is
 that now~$i$ and~$j$ play the role of $S_k-1$, and instead of~$k$, we
 want to eliminate the operators~$S_i$ and~$S_j$. Consequently we have
 to consider the $\set Q(n)[S_n]$-module which is generated by
 $\{S_i^{e_1}S_j^{e_2}|e_1,e_2=0,1,\dots\}$. 

 For our concrete application, we started with a set of 16
 annihilating operators for $f(n;i,j)$.  These operators were found by
 the ansatz described in section \ref{Guess} and verified as proposed
 in section \ref{VerifyGuess}.  The maximal degree w.r.t.\ $i$ as well
 as the maximal degree w.r.t.\ $j$ is~4. Some of the operators had
 degree less than 4 in $i$ or $j$, hence we had to add their
 corresponding multiples to the set of annihilating operators (which
 after this step consisted of 24 elements). Next we performed the
 substitution $i=0$ and $j=0$.  Finally, the elimination of $S_i$ and
 $S_j$ using Gr\"obner bases took about 30 hours and resulted in an
 operator $P(n,S_n)$ of order 32 and polynomial coefficients of degree
 172 in~$n$. (This is the operator posted on our website.)

 As already pointed out, Takayama's algorithm does not deliver $Q_1$
 and $Q_2$. However, in principle the full certificate
 $R(n,i,j,S_n,S_i,S_j)$ can be computed by doing some book-keeping
 during the run of the algorithm. But in the case of Gessel's
 conjecture this extra cost would make our computations not feasible
 with current computers (we would have to wait for a few more Moore-doublings).
 We want to emphasize that, nevertheless, the proof is \emph{completely rigorous}.

 To cite a simple (commutative) analogy, Euclid devised, more than 2300 years ago, an algorithm
 to find the greatest common divisor of two integers. Later mathematicians extended it to
 the generalized Euclidean algorithm that inputs integers $m$ and $n$ and outputs
 not only $d$, the greatest common divisor of $m$ and $n$, but also two other integers
 $a$ and $b$ such that $am+bn=d$. Just because our computers are not fast or big enough
 to actually find these two other integers $a$ and $b$ does not detract from the correctness
 of the output $d$ of the original Euclidean algorithm, and their existence is implied by it.

 \section{Conclusion}

    Computer-aided proofs have come a long way since the hostile reception of the Appel-Haken
 landmark proof of the Four-Color Conjecture. Even as recently as 1998, Hales' breakthrough
 computer-aided proof of Kepler's conjecture was met with skepticism, but it did get published,
 with some reservations, in the prestigious journal {\it Annals of Mathematics}. Both the
 Appel-Haken and Hales theorems are examples of extremely simply-stated statements, whose
 proofs defied, so far, human attepts. While Ira Gessel's conjecture has neither the
 longevity nor the notoriety of the above theorems, it does belong to that genre,
 and we believe that it is very possible that a short human proof does not exist.
 Unfortunately, formally proving this last meta-conjecture would be probably
 much more difficult than
 proving Gessel's conjecture, since proving realistic lower bounds is a notoriously difficult task.
 So we have to resort to empirical sociological testing, using the ingrained greed of human
 mathematicians.
 To that end, the third-named author (DZ) offers a prize of one hundred ($100$) US-dollars for
 a short, self-contained, human-generated (and computer-free) proof of Gessel's conjecture, not
 to exceed five standard pages typed in standard font. The longer that prize would remain unclaimed, the
 more (empirical) evidence we would have that a proof of Gessel's conjecture is indeed beyond the scope of
 humankind.


\end{document}